\documentclass[bj]{imsart}

\RequirePackage[OT1]{fontenc}
\RequirePackage{natbib}
\RequirePackage[colorlinks,citecolor=blue,urlcolor=blue]{hyperref}

\makeatletter \@addtoreset{equation}{section} \makeatother
\usepackage{graphicx}
\usepackage{amsmath}
\usepackage{amssymb}
\usepackage{amsfonts}
\usepackage{amsthm}
\usepackage{thmtools}
\usepackage{caption}
\usepackage{subcaption}

\usepackage{cleveref}
\usepackage[noend]{algpseudocode}
\usepackage{algorithm}
\usepackage[leftcaption]{sidecap}

\newcommand{\RR}{{\mathbb  R}}

\DeclareMathOperator{\var}{Var}
\DeclareMathOperator{\cov}{Cov}
\DeclareMathOperator{\Hess}{Hess}

\DeclareMathOperator{\sign}{sign}

\newcommand{\bb}[1]{\mathbb {#1}}

\newtheorem{lem}{Lemma}[section]
\newtheorem{thm}{Theorem}[section]
\newtheorem{prop}{Proposition}[section]

\newtheorem{rem}{Remark}[section]

\newtheorem{cor}{Corollary}[section]

\begin{document}

\begin{frontmatter}
\title{On the isoperimetric constant, covariance inequalities and $L_{p}$%
-Poincar\'{e} inequalities in dimension one} 
\runtitle{Covariance Inequalities} 

\begin{aug}
\author{Adrien Saumard \thanksref{t1}\ead[label=e1]{asaumard@gmail.com}}
\address{CREST, Ensai, Universit{\'e} Bretagne Loire
\\ \printead{e1}}
\thankstext{t1}{Supported by NI-AID grant 2R01 AI29168-04, and by a PIMS postdoctoral fellowship}
\author{Jon A. Wellner\thanksref{t2}\ead[label=e2]{jaw@stat.washington.edu}}
\thankstext{t2}{Supported in part by NSF Grant DMS-1566514, NI-AID grant 2R01 AI291968-04} 

\address{Department of Statistics, University
of Washington, Seattle, WA  98195-4322,\\ \printead{e2}}

\runauthor{Saumard \& Wellner}
\end{aug}

\begin{abstract}
Firstly, we derive in dimension one a new covariance inequality of $L_{1}-L_{\infty}$ 
type that characterizes the isoperimetric constant as the best constant achieving 
the inequality. Secondly, we generalize our result to $L_{p}-L_{q}$ bounds for the covariance. 
Consequently, we recover Cheeger's inequality without using the co-area formula. 
We also prove a generalized weighted Hardy type inequality that is needed to derive our covariance 
inequalities and that is of independent interest. Finally, we explore some consequences 
of our covariance inequalities for $L_{p}$-Poincar\'{e} inequalities and moment bounds. 
In particular, we obtain optimal constants in general $L_{p}$-Poincar\'{e} inequalities for measures 
with finite isoperimetric constant, thus generalizing in dimension one Cheeger's inequality, 
which is a $L_{p}$-Poincar\'{e} inequality for $p=2$, to any real $p\geq 1$.
\end{abstract}



\end{frontmatter}

\tableofcontents

\bigskip

\vspace{1cm}

\section{The isoperimetric constant}

For a measure $\mu $ on $\mathbb{R}^{d}$ (we will focus on $d=1$), an
isoperimetric inequality is an inequality of the form,%
\begin{equation}
\mu ^{+}\left( A\right) \geq c\min \left\{ \mu \left( A\right) ,1-\mu \left(
A\right) \right\} \text{ ,}  \label{isop_ineq}
\end{equation}%
where $c>0$, $A$ is an arbitrary measurable set in $\mathbb{R}^{d}$ and $\mu
^{+}\left( A\right) $ stands for the $\mu -$perimeter of $A$, defined to be%
\begin{equation*}
\mu ^{+}\left( A\right) =\lim \inf_{h\rightarrow 0^{+}}\frac{\mu \left(
A^{h}\right) -\mu \left( A\right) }{h}\text{ ,}
\end{equation*}%
where $A^{h}=\left\{ x\in \mathbb{R}^{d}:\exists a\in A,\text{ }\left\vert
x-a\right\vert <h\right\} $ is an $h$-neighborhood of $A$. The optimal value
of $c=Is\left( \mu \right) $ in (\ref{isop_ineq}) is referred to as the
\textit{Cheeger isoperimetric constant of} $\mu $. It turns out that the
isoperimetric constant is linked to the best constant in Poincar\'{e}'s
inequality, this is the celebrated Cheeger's inequality: it says that if $%
\lambda >0$ satisfies, for every smooth (i.e. locally Lipschitz) function $g$ on $\mathbb{R}^{d}$,%
\begin{equation}
\lambda \var
\left( g\right) \leq \int \left\vert \nabla g\right\vert ^{2}d\mu 
\text{ ,}  \label{poincare_ineq}
\end{equation}%
with $\var(g)=\mathbb{E}[(g-\mathbb{E}[g])^{2}]=\int \left( g-\int g d\mu \right)^{2} d\mu $
then we can take%
\begin{equation*}
\lambda \geq \left( Is\left( \mu \right) \right) ^{2}/4\text{ .}
\end{equation*}%
On $\mathbb{R}$, the isoperimetric constant achieves the following identity (%
\cite{BobHoud97}, Theorem 1.3),%
\begin{eqnarray}
Is(\mu)  = \mbox{essinf}_{a<x<b} \frac{f(x)}{\min\{ F(x), 1-F(x) \}} ,
\label{IsoperimetricConstant}
\end{eqnarray}
where $f$ is the density of the absolutely continuous part of $\mu$, $F$ is the distribution function of $\mu$, 
$a=\inf \left\{ x:F\left( x\right) >0\right\} $ and $b=\sup \left\{x:F\left( x\right) <1\right\} $.

In addition to being defined via relations on sets (\ref{isop_ineq}), the
isoperimetric constant can be stated through the use of a functional
inequality. Indeed, the isoperimetric constant is also the optimal constant
satisfying the following analytic (Cheeger-type) inequality (see for
instance p.192, \cite{BobHoud97}),%
\begin{equation}
c\int \left\vert g-\text{med}_{\mu }\left( g\right) \right\vert d\mu \leq
\int \left\vert \nabla g\right\vert d\mu \text{ ,}  \label{L1_Poin}
\end{equation}%
where $g$ is an integrable, (locally) Lipschitz function on $\mathbb{R}^{d}$ and $\text{med}_{\mu }\left( g\right)$ 
is the median of $g$ with respect to $\mu$.

Inequality (\ref{L1_Poin}) is also termed an $L_{1}$-Poincar\'{e} inequality.
Instances of $L_{1}$-Poincar\'{e} inequalities could also be considered for a
centering of $g$ by its mean rather than its median, since%
\begin{equation}
\int \left\vert g-\text{med}_{\mu }\left( g\right) \right\vert d\mu \leq
\int \left\vert g-\mathbb{E}_{\mu }\left( g\right) \right\vert d\mu \leq
2\int \left\vert g-\text{med}_{\mu }\left( g\right) \right\vert d\mu \text{ .%
}  \label{mean_median}
\end{equation}%
\cite{MR1396954}, Chapter 14, studied related optimal constants in Sobolev-type inequalities defined from Orlicz-norm.\\
In a different direction, inequality (\ref{L1_Poin}) may also be seen as a special instance of a covariance inequality.
Indeed, let us denote $\sign(x)=2\mathbf{1}_{\left\{x\geq 0 \right\}} -1$ the sign of a real number $x$. Then by definition of a median, it follows that%
\begin{equation*}
\mathbb{E}_{\mu }\left[ \sign\left( g-\text{med}_{\mu }\left( g\right)
\right) \right] =\int \sign\left( g-\text{med}_{\mu }\left( g\right)
\right) d\mu =0
\end{equation*}%
and so,%
\begin{equation}
\int \left\vert g-\text{med}_{\mu }\left( g\right) \right\vert d\mu =
\cov_{\mu }\left( g-\text{med}_{\mu }\left( g\right) ,\text{sign}\left( g-\text{%
med}_{\mu }\left( g\right) \right) \right) \text{ ,}  \label{formula_cov_L1}
\end{equation}%
where $\cov_{\mu }\left( g,h\right) =\int g(h-\mathbb{E}_{\mu }\left[ h\right] )d\mu $.

A natural question then arises: can the isoperimetric constant be defined as
the optimal constant in a covariance inequality bounding, for suitable
functions $g$ and $h$, their covariance 
$\cov_{\mu }\left( g,h\right) $ by the $L_{1}$ moment of $\nabla g,$ that is $%
\int \left\vert \nabla g\right\vert d\mu $?   
Surely, the upper-bound on the
covariance will also depend on some function of the magnitude of $\nabla h$.

Recently, \cite{Menz-Otto:2013} have established in dimension
one what they call an \textit{asymmetric Brascamp-Lieb inequality}: if $\mu $ is a
strictly log-concave measure on $\mathbb{R}$, then for smooth and square
integrable functions $g$ and $h$ on $\mathbb{R}$, 
\begin{equation}
\left\vert 
\cov_{\mu }\left( g,h\right) \right\vert \leq \left\Vert g^{\prime }\right\Vert_{1} 
\left\Vert \frac{h^{\prime }}{\varphi ^{\prime \prime }}\right\Vert_{\infty }\text{ ,}  \label{MenzOtto}
\end{equation}
where $\varphi $ is the potential of the measure $\mu $, defined by the
relation $d\mu =\exp \left( -\varphi \right) dx$, and the norms are taken with
respect to $\mu $. This result has been generalized by \cite%
{CarlenCordero-ErausquinLieb} to higher dimension and to $L_{p}-L_{q}$
versions (rather than $L_{1}-L_{\infty }$). For a strictly log-concave
measure $\mu $ on $\mathbb{R}^{d}$, $d\mu =\exp \left( -\varphi \right)
dx$, any square integrable locally Lipschitz functions $g$ and $h$ and 
$p\in \left[ 2,+\infty \right) $, $p^{-1}+q^{-1}=1$, it holds that
\begin{equation*}
\left\vert  \cov_{\mu }\left( g,h\right) \right\vert \leq \left\Vert 
\lambda _{\min}^{\left( 2-p\right) /p}\Hess_{\varphi }^{-1/p}\nabla g\right\Vert
_{q}\left\Vert \Hess_{\varphi }^{-1/p}\nabla h\right\Vert _{p}\text{ ,}
\end{equation*}
where $\lambda _{\min }\left( x\right) $ is the least eigenvalue of 
$\Hess_{\varphi \left( x\right) }$.

Such results seem close in their form to what we would want to have to
generalize (\ref{L1_Poin}). However, it is well-known (see for instance \cite%
{Ledoux:04}) that for a log-concave measure $\mu $, if $\lambda _{\min
}\left( x\right) \geq \rho >0$ - which means that $\mu $ is strongly
log-concave, see for instance \cite{SauWel2014} - then 
\begin{equation*}
\frac{Is\left( \mu \right) ^{2}}{4}\leq \lambda _{P}\leq 36Is\left( \mu
\right) ^{2}\text{ and }\rho \leq \lambda _{P}\text{ ,}
\end{equation*}
where $\lambda _{P}$ is the (optimal) Poincar\'{e} constant of $\mu $.\ In
particular, Inequality (\ref{MenzOtto}) implies in this case,%
\begin{equation}
\left\vert \cov_{\mu }\left( g,h\right) \right\vert \leq \rho ^{-1}\left\Vert g^{\prime}
\right\Vert _{1}\left\Vert h^{\prime }\right\Vert _{\infty }\text{ ,}
\label{cov_rho}
\end{equation}
for smooth and square integrable functions $g$ and $h$ on $\mathbb{R}$. 
But $\rho ^{-1}\geq (36Is\left( \mu \right)^{2})^{-1}$ and we a priori don't know if the right-hand
side of (\ref{cov_rho}) could be changed to $uIs\left( \mu \right)^{-2}\left\Vert g^{\prime
}\right\Vert _{1}\left\Vert h^{\prime }\right\Vert _{\infty }$, where $u$
would be a universal constant - note that we will give a positive answer to this question in the following. 
Hence the connection between inequalities of
the form of (\ref{cov_rho}) and the isoperimetric constant is not
straightforward.

One of the reasons for this difficulty is that if we try to approximate the
function sign$\left( g-\text{med}_{\mu }\left( g\right) \right) $ appearing
in (\ref{formula_cov_L1}) by a sequence of smooth functions $g_{n}$ in order
to use inequality (\ref{MenzOtto}) or (\ref{cov_rho}) and take the limit,
then the sequence of sup-norms 
$\left\Vert g_{n}^{\prime }\right\Vert_{\infty }$ will diverge to infinity and it is thus hopeless to recover the 
$L_{1}$-Poincar\'{e} inequality (\ref{L1_Poin}) at the limit. Another limitation of
asymmetric Brascamp-Lieb inequalities is that they hold for strictly
log-concave measures while our expected covariance inequality would ideally
be valid for any measure.

In Section \ref{section_L1_Cov}, taking the dimension to be one, we establish 
a covariance inequality that is valid for any measure on $%
\mathbb{R}$ and that indeed generalizes the $L_{1}$-Poincar\'{e} inequality (%
\ref{L1_Poin}). Then we will consider in Section \ref{section_Lp_Cov}
extensions of our covariance inequalities that are related to $L_{p}$-Poincar\'{e} inequalities, 
for $p\geq 1$. In particular, we will prove and make use of some
generalized (weighted) Hardy-type inequalities, that are of independent interest. We will explore further
consequences in terms of moment estimates of our new covariance inequalities
in Section \ref{section_moment_estimates}.

\section{A $L_{1}-L_{\infty}$ covariance inequality\label{section_L1_Cov}}

All norms, expectations and covariances will be taken with respect to a probability measure 
$\mu $ on $\mathbb{R}$ so that we will skip related indices in the notations.

Notice that if a $L_{1}$-Poincar\'{e} inequality (with a centering by the
median) holds for a measure $\mu $, that is there exists $c_{1}>0$ such
that for every smooth integrable $g$, 
\begin{equation}
c_{1}\left\Vert g-\text{med}\left( g\right) \right\Vert _{1}\leq \left\Vert
g^{\prime }\right\Vert _{1}\text{ ,}
\label{ineq_cheeger_dim1}
\end{equation}
then if $h$ is in $L_{\infty }$, we get
\begin{eqnarray}
\left\vert  \cov \left( g,h\right) \right\vert =\left\vert \mathbb{E}\left[ 
\left( g-\text{med}\left( g\right) \right) \left( h-\mathbb{E}\left[ h\right] \right) \right] \right\vert   
\leq \left\Vert g-\text{med}\left( g\right) \right\Vert _{1}\left\Vert h-%
\mathbb{E}\left[ h\right] \right\Vert _{\infty } \notag \\
\leq  c_{1}^{-1}\left\Vert
g^{\prime }\right\Vert _{1}\left\Vert h_{0}\right\Vert _{\infty }\text{ ,}
\label{ineq_cov_L1_Poin}
\end{eqnarray}

where $h_{0}:=h-\mathbb{E}\left[ h\right] $. Moreover, the optimal constant
in (\ref{ineq_cheeger_dim1}) is $c_{1}=Is\left( \mu \right) $.

The following theorem states an inequality in dimension one that is sharper in terms of the control of $h$.

\begin{thm}
\label{th_cov_isop copy(1)}Let $\mu $ be a probability measure with a
positive density $f$ on $\mathbb{R}$, cumulative distribution $F$ and median 
$m\in \mathbb{R}$. Let $g\in L_{\infty }\left( F\right) $ and $h\in L_{1}\left( F\right) $. 
Assume also that $g$ and $h$ are absolutely continuous. Then we have,
\begin{eqnarray}
&& \left\vert 
\cov \left( g,h\right) \right\vert \notag \\
& \leq & \label{ineq_cov_L1_Linfty} Is\left( \mu \right) ^{-1}\max \left\{ \sup_{a<x\leq m}\left\vert 
\frac{\int_{a}^{x}hdF}{F\left( x\right) }-\mathbb{E}\left[ h\right]
\right\vert ,\text{ }\sup_{b>x>m}\left\vert \frac{\int_{x}^{b}hdF}{1-F\left(
x\right) }-\mathbb{E}\left[ h\right] \right\vert \right\} \int_{\mathbb{R}
}\left\vert g^{\prime }\right\vert dF \qquad
\end{eqnarray}
where $a=\inf \left\{ x:F\left( x\right) >0\right\} $ and $b=\sup \left\{
x:F\left( x\right) <1\right\} .$
\end{thm}

Before proving Theorem \ref{th_cov_isop copy(1)}, let us recall a
representation formula for the covariance of two functions which first appeared in  
\cite{Menz-Otto:2013} and was further studied in \cite{saumwellner2017efron}.
We define a non-negative and symmetric kernel $K_{\mu }$ on $\mathbb{R}^2$ by  
\begin{equation}
K_{\mu }\left(x,y\right) =F\left( x\wedge y\right) -F\left( x\right) F\left( y\right), 
\qquad  \mbox{for all} \left( x,y\right) \in \mathbb{R}^{2} .
\label{def_kernel}
\end{equation}
where 
$F\left( x\right) =F_{\mu }\left( x\right) =\mu \left( \left( -\infty ,x\right] \right) $ 
is the  distribution function associated with the 
probability measure $\mu $ on $(\mathbb{R}, {\cal B})$.

\begin{lem}[Corollary 2.2, \cite{saumwellner2017efron}]
\label{prop_cov_rep}
If $g$ and $h$ are absolutely continuous and 
$g\in L_{p}(F)$, $h\in L_{q}(F)$ for some $p\in \lbrack 1,\infty ]$ and $%
p^{-1}+q^{-1}=1$, then 
\begin{equation}
\cov(g,h)
=\int \!\!\!\int_{\mathbb{R}^{2}}g^{\prime }(x)K_{\mu}(x,y)h^{\prime }(y)dxdy \text{.}
\label{cov_id_kernel}
\end{equation}
\end{lem}

In fact, Lemma \ref{prop_cov_rep} can be seen as a special instance of a 
covariance representation lemma due to \citep{MR0004426} (see also \citep{MR1307621} 
for a translation of the German original paper into English). 
Hoeffding's covariance representation is as follows: Let $X$ and $Y$ be two
real random variables with finite second moments. Then%
\begin{equation*}
\cov \left( X,Y\right) =\int \! \!\! \int K_{\left( X,Y\right) }\left( x,y\right) dxdy
\text{ ,}
\end{equation*}
with $K_{\left( X,Y\right) }\left( x,y\right) =F_{\left( X,Y\right) }\left(
x,y\right) -F_{X}\left( x\right) F_{Y}\left( y\right) =\mathbb{P}\left(
X\leq x,Y\leq y\right) -\mathbb{P}\left( X\leq x\right) \mathbb{P}\left(
Y\leq y\right) $.

For a real random variable $Z$ with probability distribution $\mu$ on $\RR$, 
take $X=g\left( Z\right) $ and $Y=h\left(Z\right) $, 
where $g$ and $h$ are nondecreasing and left-continuous. Then by
Hoeffding's covariance representation,
\begin{equation*}
\cov 
\left( g\left( Z\right) ,h\left( Z\right) \right) 
=\int \!\!\! \int K_{\left(
g\left( Z\right) ,h\left( Z\right) \right) }\left( x,y\right) dxdy\text{ .}
\end{equation*}
Furthermore,
\begin{equation*}
K_{\left( g\left( Z\right) ,h\left( Z\right) \right) }\left( x,y\right) =
\mathbb{P}\left( g\left( Z\right) \leq x,h\left( Z\right) \leq y\right) -
\mathbb{P}\left( g\left( Z\right) \leq x\right) \mathbb{P}\left( h\left(Z\right) \leq y\right)
\end{equation*}
and by considering the generalized inverses $g^{-1}\left( x\right) =\inf\left\{ z:g\left( z\right) \geq x\right\} $ 
and $h^{-1}$  of $g$ and $h$ respectively, it follows that
\begin{eqnarray*}
K_{\left( g\left( Z\right) ,h\left( Z\right) \right) }\left( x,y\right) 
&=&
\mathbb{P}\left( Z\leq g^{-1}\left( x\right) ,Z\leq h^{-1}(y)\right) -
\mathbb{P}\left( Z\leq g^{-1}\left( x\right) \right) \mathbb{P}\left( Z\leq h^{-1}(y)\right) \\
&=&K_{\left( Z,Z\right) }\left( g^{-1}\left( x\right) ,h^{-1}(y)\right)
=:K_{Z}\left( g^{-1}\left( x\right) ,h^{-1}(y)\right) \text{ .}
\end{eqnarray*}
where $K_Z = K_{\mu}$.
Hence, by change of variables,
\begin{eqnarray*}
\cov
\left( g\left( Z\right) ,h\left( Z\right) \right) 
= \int \!\!\! \int K_{Z}\left( g^{-1}\left( x\right) ,h^{-1}(y)\right) dxdy 
= \int \!\!\! \int K_{Z}\left( u,v\right) dg\left( u\right) dh(v)\text{ .}
\end{eqnarray*}
Considering differences of monotone functions and restricting to absolutely
continuous functions now gives Menz and Otto's covariance identity: 
for $g\left( Z\right) $ and $h\left( Z\right) $ in $L_{2}$ and absolutely continuous,%
\begin{equation*}
\cov \left( g\left( Z\right) ,h\left( Z\right) \right) 
=\int \!\!\! \int g^{\prime}\left( x\right) K_{Z}\left( x,y\right) h^{\prime }\left( y\right) dxdy\text{.}
\end{equation*}

Hoeffding's covariance representation has been extended to a multi-dimensional setting, 
by considering the so-called cumulant between several random variables rather 
than the covariance (\cite{MR958217}). However, it remains unclear how such 
extension of Hoeffding's lemma could be used to generalize our results, 
for instance concerning $L_p$-Poincar\'{e} inequalities, to a multivariate setting.
It is also worth mentioning that the fact that the covariance representation (\ref{cov_id_kernel}) 
of Lemma \ref{prop_cov_rep} is valid for any probability measure $\mu$ is specific to dimension one. 
More precisely, \cite{MR1836739} proved that, in dimension greater than two, 
a covariance identity of the form of (\ref{cov_id_kernel}) 
implies that $\mu$ is Gaussian. 
\cite{MR1836739} also proved some genuine concentration inequalities from 
such covariance identities. Further extensions of covariance inequalities and 
related deviation inequalities for infinitely divisible and stable random vectors 
have been obtained in \cite{MR1920106} and \cite{MR2060306}.

We will also need the following formulas, which are in fact special 
instances of the previous covariance representation formula.

\begin{lem}[Corollary 2.1, \cite{saumwellner2017efron}]
\label{prop_formulas_kernel}
For an absolutely continuous function $h\in L_{1}(F)$, 
\begin{equation}
F(z)\int_{\mathbb{R}}hd F-\int_{-\infty }^{z}hdF
=\int_{\mathbb{R}}K_{\mu }\left( z,y\right) h^{\prime }(y)dy  
\label{rep_g_prime_L1-MRLa}
\end{equation}
and
\begin{equation}
- (1-F(z)) \int_{\mathbb{R}}hd F + \int_{(z,\infty)} h dF
=\int_{\mathbb{R}}K_{\mu }\left( z,y\right) h^{\prime }(y)dy . 
\label{rep_g_prime_L1-MRLb}
\end{equation}
\end{lem}

We are now able to give a proof of Theorem \ref{th_cov_isop copy(1)}.
\begin{proof}
Using the notation of Theorem \ref{th_cov_isop copy(1)},  Lemma \ref{prop_cov_rep} yields
%above we get 
\begin{eqnarray*}
\lefteqn{\left\vert \cov\left( g,h\right) \right\vert} \\
&=&\left\vert \int
\!\!\int g^{\prime }\left( x\right) K_{\mu }\left( x,y\right) h^{\prime
}\left( y\right) dxdy\right\vert \\
&=&\left\vert \int g^{\prime }\left( x\right) \left( \int K_{\mu }\left(
x,y\right) h^{\prime }\left( y\right) dy\right) dx\right\vert \\
&\leq &\int \left\vert g^{\prime }\left( x\right) \right\vert \left\vert
\int K_{\mu }\left( x,y\right) h^{\prime }\left( y\right) dy\right\vert dx \\
&=&\int \left\vert g^{\prime }\left( x\right) \right\vert \frac{\left\vert
\int K_{\mu }\left( x,y\right) h^{\prime }\left( y\right) dy\right\vert }{%
f\left( x\right) }f\left( x\right) dx \\
& \leq & \sup_{a<x\le m} \frac{| \int_{\RR} K_{\mu} (x,y) h'(y) dy|}{f(x)} \int_{(-\infty,m]} |g'| d F  
           +  \sup_{m<x<b} \frac{| \int_{\RR} K_{\mu} (x,y) h'(y) dy|}{f(x)} \int_{(m, \infty)} |g'| d F \text{ .}
\end{eqnarray*}%
Now, by using (\ref{rep_g_prime_L1-MRLa}) and (\ref{rep_g_prime_L1-MRLb}), we get
\begin{eqnarray*}
\lefteqn{\sup_{a<x\le m} \frac{| \int_{\RR} K_{\mu} (x,y) h'(y) dy|}{f(x)} \int_{(-\infty,m]} |g'| d F 
 \ + \ \sup_{m<x<b} \frac{| \int_{\RR} K_{\mu} (x,y) h'(y) dy|}{f(x)} \int_{(m, \infty)} |g'| d F  } \\
& \le & \max \left \{ \sup_{a<x\le m} \frac{| \frac{\int_{-\infty}^x h dF}{F(x)} - \mathbb{E}\left[ h\right]|}{f(x)/F(x)} \ , \  
                       \sup_{m<x<b} \frac{| \frac{\int_{(x, \infty)} h dF }{1-F(x)}- \mathbb{E}\left[ h\right]|}{f(x)/(1-F(x))} \right \} \int_{\RR} | g' | dF \\
%& \le & \max \left \{ \sup_{x\le m} \frac{| \frac{\int_{-\infty}^x h dF}{F(x)} - Eh(X)|}{f(x)/F(x)} \ , \  
%                       \sup_{x\ge m} \frac{| \frac{\int_{(x, \infty)} h dF }{1-F(x)}- Eh(X)|}{f(x)/(1-F(x))} \right \} \int_{\RR} | g' | dF \\
& \le & (Is(\mu))^{-1}\max \left \{ \sup_{a<x \le m} \left\vert \frac{\int_{-\infty}^x h dF}{F(x)} - \mathbb{E}\left[ h\right]\right\vert \ , \  
                       \sup_{m<x<b} \left\vert \frac{\int_{(x, \infty)} h dF }{1-F(x)}- \mathbb{E}\left[ h\right]\right\vert \right \} \int_{\RR} | g' | dF \text{ .}
\end{eqnarray*}     
by using (\ref{IsoperimetricConstant}) in the last line.               
\end{proof}

\begin{rem}
The proof of Theorem \ref{th_cov_isop copy(1)} allows to give other variants of covariance inequalities. 
Indeed, we have 
\begin{eqnarray*}
\lefteqn{\left\vert \cov\left( g,h\right) \right\vert \leq \int \left\vert g^{\prime }\left( x\right) \right\vert \frac{\left\vert
\int K_{\mu }\left( x,y\right) h^{\prime }\left( y\right) dy\right\vert }{%
f\left( x\right) }f\left( x\right) dx} \\
&\leq &\sup_{x\in \mathbb{R}}\left\{ \frac{\left\vert \int K_{\mu }\left(
x,y\right) h^{\prime }\left( y\right) dy\right\vert }{f\left( x\right) }%
\right\} \int_{\mathbb{R}}\left\vert g^{\prime }\left( x\right) \right\vert
f\left( x\right) dx\text{ .}
\end{eqnarray*}
Then, using Lemma \ref{prop_formulas_kernel}, we get
\begin{eqnarray}
%\lefteqn{
\left\vert 
\cov\left( g,h\right) \right\vert %}  \notag \\
\leq  \left \{ 
             \begin{array}{l}
             \sup_{x\in (a,b)}\left\{ \frac{\left\vert F\left( x\right) 
             \int_{\mathbb{R}}h dF -\int_{\mathbb{-\infty }}^{x}h dF \right\vert }{f\left( x\right) } \right\} 
             \int_{\mathbb{R}}\left\vert g^{\prime } \right\vert dF \text{}    \\
             \sup_{x\in (a,b)}\left\{ \frac{\left\vert (1-F\left( x\right)) 
             \int_{\mathbb{R}}h dF -\int_{(x,\infty)} hdF \right\vert }{f\left( x\right) } \right\} 
             \int_{\mathbb{R}}\left\vert g^{\prime } \right\vert dF
           \end{array} \right .   \label{cov_bound_th} 
\end{eqnarray}
The latter covariance inequality generalizes Menz and Otto's covariance inequality 
(\ref{MenzOtto}) for strictly log-concave measures to any measure with positive density 
on the real line. Indeed, let us write $f\left( x\right) =\exp \left( -\varphi \left( x\right) \right) $,
with $\varphi :\mathbb{R\rightarrow R}$. If  $f$ is strictly log-concave, 
and $\varphi ^{\prime }$ is absolutely continuous, so that $\varphi $
is convex and $\varphi ^{\prime \prime }>0$ on $\mathbb{R}$, then by
Lemma \ref{prop_formulas_kernel} and Corollary 2.3 of \cite{saumwellner2017efron}, we have
\begin{eqnarray*}
\lefteqn{\sup_{x\in \mathbb{R}}\left\{ \frac{\left\vert F\left( x\right) 
           \int_{\mathbb{R}}h\left( y\right) f\left( y\right) dy-\int_{\mathbb{-\infty }}^{x}h\left( y\right) 
           f\left( y\right) dy\right\vert }{f\left( x\right) } \right\} }\\
&=& \sup_{x\in \mathbb{R}}\left\{ \frac{\left\vert 
        \int K_{\mu }\left(x,y\right) h^{\prime }\left( y\right) dy\right\vert }{f\left( x\right) } \right\} \\
&=&\sup_{x\in \mathbb{R}}\left\{ \frac{\left\vert \int K_{\mu }\left(x,y\right) \frac{h^{\prime }\left( y\right) }{\varphi ^{\prime \prime}
        \left( y\right) }\varphi ^{\prime \prime }\left( y\right) dy\right\vert }{f\left( x\right) }\right\} \\
&\leq & \sup_{x\in \mathbb{R}}\left\{ \frac{\sup_{y\in \mathbb{R}}\left\vert 
          \frac{h^{\prime }\left( y\right) }{\varphi ^{\prime \prime }\left( y\right) } \right\vert 
         \int K_{\mu }\left( x,y\right) \varphi ^{\prime \prime }\left( y\right) dy}{f\left( x\right) }\right\} \\
&=& \sup_{y\in \mathbb{R}}\left\vert \frac{h^{\prime }\left( y\right) }{ \varphi ^{\prime \prime }\left( y\right) }\right\vert \text{.}
\end{eqnarray*}
Applying the latter bound in inequality (\ref{cov_bound_th}) indeed yields the asymmetric Brascamp - Lieb inequality
presented in \cite{Menz-Otto:2013}. 
Thus the covariance inequality (\ref{cov_bound_th}) 
may be viewed as a generalization of Menz and Otto's result in dimension one.
\end{rem}

\bigskip 

By setting, for $k\in \overline{\RR} $,%
\begin{equation*}
T_{k}h\left( x\right) =\mathbf{1}_{\left( a,k\right] }\left( x\right) \frac{1%
}{F\left( x\right) }\int_{\left( a,k\right] }hdF+\mathbf{1}_{\left(
k,b\right) }\left( x\right) \frac{1}{1-F\left( x\right) }\int_{\left(
k,b\right) }hdF\text{ ,}
\end{equation*}%
Inequality (\ref{ineq_cov_L1_Linfty}) of Theorem \ref{th_cov_isop copy(1)}\
yields,%
\begin{equation*}
\left\vert 
%TCIMACRO{\TeXButton{Cov}{\cov}}%
%BeginExpansion
\cov%
%EndExpansion
\left( g,h\right) \right\vert \leq Is\left( \mu \right) ^{-1}\left\Vert
g^{\prime }\right\Vert _{1}\left\Vert T_{m}h_{0}\right\Vert _{\infty }\text{
.}
\end{equation*}%
It is straightforward to see that $\left\Vert T_{m}h_{0}\right\Vert _{\infty
}\leq \left\Vert h_{0}\right\Vert _{\infty }$, so that we recover inequality
(\ref{ineq_cov_L1_Poin}). 

We will say that a measure $\mu $ satisfies a $L_{1}-L_{\infty}$ covariance
inequality with constant $r>0$, if, for
every $g\in L_{\infty }\left( F\right) $ and $h\in L_{1}\left( F\right) $
with $g$ and $h$  absolutely continuous, %follows that%
\begin{equation}
\left\vert 
%TCIMACRO{\TeXButton{Cov}{\cov}}%
%BeginExpansion
\cov%
%EndExpansion
_{\mu }\left( g\left( X\right) ,h\left( X\right) \right) \right\vert \leq
r\left\Vert g^{\prime }\right\Vert _{1}\left\Vert T_{m}h_{0}\right\Vert
_{\infty }\text{ .}  \label{extended_Poin}
\end{equation}%
In this case, we denote $\cov_{1,\infty}\left( \mu \right) $ the smallest 
constant achieving a $L_{1}-L_{\infty}$ covariance inequality 
for the measure $\mu$. 
In other words, if (\ref{extended_Poin}) is valid for
every $g\in L_{\infty }\left( F\right) $ and $h\in L_{1}\left( F\right) $, with $g
$ and $h$ absolutely continuous, then $\cov_{1,\infty}\left( \mu \right) \leq r$.
We have the following optimality result.

\begin{prop}
\label{theorem_eP}Let $\mu $ be a probability measure with a
positive density $f$ on $\mathbb{R}$ and cumulative distribution $F$. If $%
\mu $ satisfies a $L_{1}-L_{\infty}$-covariance inequality then its isoperimetric constant is finite and%
\begin{equation*}
\frac{1}{Is\left( \mu \right) }\leq \cov_{1,\infty}\left( \mu \right)\text{ .}
\end{equation*}%
Furthermore, if $\mu $ has a finite isoperimetric constant,  then 
$\mu $ satisfies a $L_{1}-L_{\infty}$-covariance inequality with constant $%
Is\left( \mu \right) ^{-1}$. In other words, the inverse of the isoperimetric constant $%
Is\left( \mu \right) ^{-1}$ is the optimal constant achieving inequality (%
\ref{extended_Poin}) when available,
\begin{equation*}
\frac{1}{Is\left( \mu \right) }= \cov_{1,\infty}\left( \mu \right)\text{ .}
\end{equation*}%
\end{prop}

\begin{proof}
The second part simply corresponds to Theorem \ref{th_cov_isop copy(1)}. For the
first part, we will use Lemma \ref{lemma_bob_hou}\ below that is due to \cite{BobHoud97}. 
Indeed, consider an absolutely continuous function $g\in
L_{1}\left( \mu \right) $. As the function sign$\left( g-\text{med}\left(
g\right) \right) =\mathbf{1}_{\left\{ g-\text{med}\left( g\right) \geq
0\right\} }-\mathbf{1}_{\left\{ g-\text{med}\left( g\right) \leq 0\right\} }$%
, we deduce from Lemma \ref{lemma_bob_hou} that there exists a sequence of
Lipschitz functions $h_{n}$ on $\mathbb{R}$ with values in $\left[ -1,1%
\right] $ such that $h_{n}\rightarrow $sign$\left( g-\text{med}\left(
g\right) \right) $ pointwise as $n\rightarrow \infty $. Hence,  the
dominated convergence theorem and identity (\ref{formula_cov_L1}) give%
\begin{equation*}
%TCIMACRO{\TeXButton{Cov}{\cov}}%
%BeginExpansion
\cov%
%EndExpansion
_{\mu }\left( g\left( X\right) ,h_{n}\left( X\right) \right) \rightarrow
 \mathbb{E}\left\vert g-\text{med}\left( g\right)
\right\vert \text{  as  } n\rightarrow \infty  \text{ .}
\end{equation*}%
Furthermore, again by dominated convergence,%
\begin{eqnarray*}
\lefteqn{\max \left\{ \sup_{a<x\leq m}\left\vert \frac{\int_{-\infty }^{x}h_{n}dF}{%
F\left( x\right) }-\mathbb{E}\left[ h_{n}\right] \right\vert ,\text{ }%
\sup_{b>x>m}\left\vert \frac{\int_{x}^{+\infty }h_{n}dF}{1-F\left( x\right) }%
-\mathbb{E}\left[ h_{n}\right] \right\vert \right\} } \\
&\leq &\left\Vert h_{n}-\mathbb{E}\left[ h_{n}\right] \right\Vert _{\infty
}\leq 1+\left\vert \mathbb{E}\left[ h_{n}\right] \right\vert \rightarrow
1 \ \ \ \mbox{as} \ \ n \rightarrow \infty .
\end{eqnarray*}%
Now, the conclusion simply follows from (\ref{L1_Poin}).
\end{proof}

\begin{rem}
In dimension $d\geq 1$, if a measure $\mu $ has a finite isoperimetric
constant, Inequality (\ref{L1_Poin}) combined with H\"{o}lder's inequality implies
that for any $g\in L_{1}\left( \mu \right) $ locally Lipschitz and any $h\in L_{\infty }\left( \mu \right) $,%
\[
\left\vert 
%TCIMACRO{\TeXButton{Cov}{\cov}}%
%BeginExpansion
\cov%
%EndExpansion
\left( g,h\right) \right\vert \leq \left( Is\left( \mu \right) \right)
^{-1}\left\Vert \nabla g\right\Vert _{1}\left\Vert h_{0}\right\Vert _{\infty
}\text{ .}
\]%
Now, if a measure $\mu $ satisfies for any $g\in L_{1}\left( \mu \right) $
locally Lipschitz and any $h\in L_{\infty }\left( \mu \right) $,%
\[
\left\vert 
%TCIMACRO{\TeXButton{Cov}{\cov}}%
%BeginExpansion
\cov%
%EndExpansion
\left( g,h\right) \right\vert \leq r\left\Vert \nabla g\right\Vert
_{1}\left\Vert h_{0}\right\Vert _{\infty }\text{ ,}
\]
for some finite constant $r>0$, then by the same arguments as
in the proof of Proposition \ref{theorem_eP} and in particular by the use of 
Lemma \ref{lemma_bob_hou}, $\mu$ has a finite isoperimetric constant.
\end{rem}

\begin{lem}[\protect\cite{BobHoud97}, Lemma 3.5]
\label{lemma_bob_hou}For any Borel set $A\subset X$ with 
$0<\mu \left( A\right) <1$, there exists a sequence of Lipschitz functions $f_{n}$ on $%
\mathbb{R}$ with values in $\left[ 0,1\right] $ such that $f_{n}\rightarrow 
\mathbf{1}_{\text{clos}\left( A\right) }$ pointwise as $n\rightarrow \infty $,%
and 
$\lim \sup_{n\rightarrow \infty }\bb{E}\left\vert f_{n}^{\prime}\right\vert \leq \mu ^{+}\left( A\right) $.
\end{lem}

\section{$L_{p}-L_{q}$ covariance inequalities and $L_p$-Poincar\'{e} inequalities \label{section_Lp_Cov}}
Let us begin this section by deriving the following $L_{p}-L_{q}$ 
covariance inequalities, that generalize Theorem \ref{th_cov_isop copy(1)}.

\begin{thm}
\label{theorem_lp_lq}Let $\mu $ be a probability measure with a positive
density $f$ on $\mathbb{R}$ and cumulative distribution $F$. Let $g\in L_{p}\left( F\right) $
and $h\in L_{q}\left( F\right) ,$ $p^{-1}+q^{-1}=1$, $p \in [1,+\infty)$. Assume also that $g$ and $h$ are
absolutely continuous. Then we have,%
\begin{equation}
\left\vert 
%TCIMACRO{\TeXButton{Cov}{\cov}}%
%BeginExpansion
\cov%
%EndExpansion
\left( g,h \right) \right\vert \leq 
\left[ Is\left( \mu \right) \right] ^{-1}\left\Vert g^{\prime }\right\Vert
_{p}\left\Vert T_m(h_{0})\right\Vert _{q}\text{ .}  \label{ineq_cov_lpTmlq}
\end{equation}
Consequently, we also have
\begin{equation}
\left\vert 
%TCIMACRO{\TeXButton{Cov}{\cov}}%
%BeginExpansion
\cov%
%EndExpansion
\left( g,h \right) \right\vert \leq p%
\left[ Is\left( \mu \right) \right] ^{-1}\left\Vert g^{\prime }\right\Vert
_{p}\left\Vert h_{0}\right\Vert _{q}\text{ .}  \label{ineq_cov_lplq}
\end{equation}
\end{thm}

Before proving Theorem \ref{theorem_lp_lq}, we note that Inequality 
(\ref{ineq_cov_lplq}) is a consequence of Inequality (\ref{ineq_cov_lpTmlq}) 
applied together with the following weighted Hardy type inequality, that is of independent interest.
\begin{thm}[Generalized Hardy Inequality]
\label{theorem_hardy}
Let $F$ be a continuous distribution on $\mathbb{R}$, with $a=\inf \left\{ x:F\left( x\right) >0\right\} $ and $b=\sup \left\{
x:F\left( x\right) <1\right\} .$ For a function $h\in
L_{p}\left( F\right) ,$ $1< p < \infty$ and $k \in \overline{\RR}$, 
\begin{equation}
\label{ineq_Hardy_Tm}
\| T_k(h) \|_{p}^p =\int_{a<x\leq k}\left\vert \frac{\int_{-\infty }^{x}hdF}{F\left( x\right) }%
\right\vert ^{p}dF(x)+\int_{b>x\geq k}\left\vert \frac{\int_{x}^{\infty }hdF}{%
1-F\left( x\right) }\right\vert ^{p}dF(x)\leq \left( \frac{p}{p-1}\right)
^{p}\left\Vert h\right\Vert _{p}^{p}\text{ .}
\end{equation}
In particular, 
\begin{eqnarray*}
\int_{\RR} \bigg | \frac{1}{F(x)} \int_{(-\infty,x]} h(y) dF(y) \bigg |^p dF(x) 
\le  \left ( \frac{p}{p-1} \right )^p \int_{\RR} | h(x)|^p dF(x) 
\end{eqnarray*}
and 
\begin{eqnarray*}
&& \int_{\RR} \bigg | \frac{1}{1-F(x)} \int_{(x,\infty)} h(y) dF(y) \bigg |^p dF(x) 
      \le  \left ( \frac{p}{p-1} \right )^p \int_{\RR} | h(x)|^p dF(x) .
\end{eqnarray*}
\end{thm}
The proof of Theorem \ref{theorem_hardy} can be found below. 
The fact that Hardy-type inequalities naturally come into play here is appealing, 
since it is well-known from the work of \cite{MR1710983} and \cite{MR1682772} - see also \cite{MR1845806}, 
Chapter 6 and \cite{MR2146071}, Chapters 4 and 5 - that Hardy-type inequalities 
can be used to have access to sharp constants in Poincar\'{e} and log-Sobolev inequalities 
on the real line and also in the discrete setting.\\
It is also worth noting that Inequality (\ref{ineq_cov_lplq}) of Theorem \ref{theorem_lp_lq} 
induces the celebrated Cheeger's inequality as a corollary and thus in particular 
gives a new proof of it, avoiding the classical use of the co-area formula. 

\begin{cor}[Cheeger's inequality]
\label{corollary_cheeger}
Let $\mu $ be a probability measure with a positive density $f$ on $\RR$ and
cumulative distribution $F$. Let $g\in L_{2}\left( F\right) $, absolutely continuous. Then we
have,%
\begin{equation}
%TCIMACRO{\TeXButton{Var}{\var}}%
%BeginExpansion
\var%
%EndExpansion
\left( g \right) \leq 4\left[ Is\left( \mu \right) %
\right] ^{-2}\left\Vert g^{\prime }\right\Vert _{2}^{2}\text{ .}
\end{equation}%
Consequently, if  $\lambda _{1}$ denotes the best constant in the Poincar\'{e}
inequality, it follows that%
\begin{equation*}
\left( Is\left( \mu \right) \right) ^{2}/4\leq \lambda _{1}\text{ .}
\end{equation*}
\end{cor}

\begin{proof} Simply take $g=h$ and $p=2$ in Inequality (\ref{ineq_cov_lplq}).
\end{proof}

\begin{proof}[Proof of Theorem \protect\ref{theorem_lp_lq}]
The case where $p=1$ and $q=+\infty$ is given by Theorem \ref{th_cov_isop copy(1)}. 
Let us assume that $p,q \in (1,+\infty)$. We have%
\begin{eqnarray*}
\lefteqn{\left\vert 
%TCIMACRO{\TeXButton{Cov}{\cov}}%
%BeginExpansion
\cov%
%EndExpansion
( g ,h) \right\vert } \\
&\leq &\int_{(a,b)}\left\vert g^{\prime }\left( x\right) \right\vert 
\frac{\left\vert \int_{\mathbb{R}}K_{\mu }\left( x,y\right) h^{\prime}\left(
y\right) dy\right\vert } {f\left( x\right) }f\left( x\right) dx \\
&=& \int_{(a,b)}\left\vert g^{\prime }\left( x\right) \right\vert
\left( \frac{\left\vert F\left( x\right) \int_{\mathbb{R}}hdF -\int_{-\infty
}^{x}hdF\right\vert }{f\left( x\right) } 1_{x\leq m} +\frac{\left\vert
\left( 1-F\left( x\right) \right) \int_{\mathbb{R}}hdF-\int_{x}^{\infty
}hdF\right\vert } {f\left( x\right) }1_{x \ge m}\right) f\left( x\right) dx
\\
&\leq &\left\Vert g^{\prime }\right\Vert _{p}\left( \int_{a<x\leq m} \frac{%
\left\vert F\left( x\right) \int_{\mathbb{R}}hdF
-\int_{-\infty}^{x}hdF\right\vert ^{q}} {f^{q}\left( x\right) }dF
+\int_{b>x\geq m} \frac{\left\vert \left( 1-F\left( x\right) \right) \int_{%
\mathbb{R} }hdF -\int_{x}^{\infty }hdF\right\vert ^{q}} {f^{q}\left(
x\right) }dF\right)^{1/q} \\
&=&\left\Vert g^{\prime }\right\Vert _{p}\left( \int_{a<x\leq m} \frac{%
\left\vert \int_{\mathbb{R}}hdF-\frac{\int_{-\infty }^{x}hdF} {%
F\left(x\right) }\right\vert ^{q}}{f^{q}\left( x\right) /F^{q}\left(
x\right) } dF + \int_{b>x\geq m}\frac{\left\vert \int_{\mathbb{R}}hdF-\frac{%
\int_{x}^{\infty }hdF}{1-F\left( x\right) }\right\vert ^{q}} {%
f^{q}\left(x\right) /\left( 1-F\left( x\right) \right) ^{q}}dF\right) ^{1/q}
\\
&\leq & Is\left( \mu \right) ^{-1}\left\Vert g^{\prime }\right\Vert _{p}
\left\Vert T_m (h_{0})\right\Vert _{q}\text{ .}
\end{eqnarray*}
Hence, inequality (\ref{ineq_cov_lpTmlq}) is proved. 
To prove inequality (\ref{ineq_cov_lplq}), simply combine inequality (\ref{ineq_cov_lpTmlq}) 
 with inequality (\ref{ineq_Hardy_Tm}) of Theorem \ref{theorem_hardy}.
\end{proof}

\begin{proof}[Proof of Theorem \ref{theorem_hardy}]
It suffices to prove the inequalities for $h\geq 0$. From now on, we assume
that $h\geq 0$. To prove Inequality (\ref{ineq_Hardy_Tm})  we first define the functions 
$\delta _{x}$ and $F_{x,s}$:  for $x\in (a,b)$, 
\begin{equation*}
\delta _{x}\left( t\right) := 
1_{[x,\infty)}(k)\frac{1_{(-\infty ,x]}(t)}{F(x)}  + 
    1_{(-\infty,x)}(k)\frac{1_{(x,\infty )}(t)}{1-F(x)},\ \ \mbox{for}\ \ t\in \mathbb{R}, 
%:=1_{(-\infty ,x]}(k)\frac{1_{(-\infty ,x]}(t)}{F(x)}  + 
%    1_{(x,\infty )}(k)\frac{1_{(x,\infty )}(t)}{1-F(x)},\ \ \mbox{for}\ \ t\in \mathbb{R},
\end{equation*}
and, for $t\in \mathbb{R}$, $s>0$,
\begin{equation*}
F_{x,s}\left( t\right) 
:=  F^{s}\left( t\right)   1_{[x,\infty)}(k) 1_{(-\infty ,x]}(t)
     \ +\ (1-F(t))^{s} 1_{(-\infty,x)}(k) 1_{(x,\infty )}(t) .
\end{equation*}
It follows that 
\begin{eqnarray*}
\| T_k h \|_{p}^p 
& = & \int_{a<x\leq k}\left( \frac{\int_{-\infty }^{x}hdF}{F\left( x\right) }\right)^{p}dF(x)
    +\int_{b>x > k}\left( \frac{\int_{x}^{\infty }hdF}{1-F\left( x\right) }\right) ^{p}dF(x)\\
& =  & \int_{x\in \mathbb{R}}\left( \int_{t\in \mathbb{R}}h\left( t\right) \delta _{x}
     \left( t\right) dF\left( t\right) \right)^{p}dF\left( x\right) \text{ .}
\end{eqnarray*}
Now, by H\"{o}lder's inequality,
\begin{eqnarray*}
\lefteqn{\left( \int_{t\in \mathbb{R}}h\left( t\right) \delta _{x}\left(
t\right) dF\left( t\right) \right) ^{p}} \\
&\leq &\int_{t\in \mathbb{R}}\left[ h\left( t\right) \delta _{x}\left(
t\right) F_{x,s}\left( t\right) \right] ^{p}dF\left( t\right) \left(
\int_{u\in \mathbb{R}}\left[ F_{x,s}\left( u\right) \right]^{-q}dF\left( u\right) \right) ^{p/q}\text{ .}
\end{eqnarray*}
The use of Fubini's theorem then gives
\begin{eqnarray*}
\lefteqn{\int_{x\in \mathbb{R}}\left( \int_{t\in \mathbb{R}}h\left( t\right)
              \delta _{x}\left( t\right) dF\left( t\right) \right) ^{p}dF\left( x\right) } \\
&\leq &  \int_{x\in \mathbb{R}}\int_{t\in \mathbb{R}}\left[ h\left( t\right) \delta _{x}\left( t\right) F_{x,s}\left( t\right) \right] ^{p}dF(t)
              \left( \int_{u\in \mathbb{R}}\left[ F_{x,s}\left( u\right) \right] ^{-q}dF\left( u\right) \right) ^{p/q}dF\left( x\right)  \\
& = &\int_{t\in \mathbb{R}}h^{p}\left( t\right) \left( \int_{x\in \mathbb{R}} \left[ \delta _{x}\left( t\right) F_{x,s}\left( t\right) \right]^{p}
           \left( \int_{u\in \mathbb{R}}\left[ F_{x,s}\left( u\right) \right]^{-q}dF\left( u\right) \right) ^{p/q}dF\left( x\right) \right) 
           dF\left( t\right) \text{ .}
\end{eqnarray*}
In order to conclude, it suffices to prove that for an appropriate choice of $s>0$,
\begin{equation*}
\int_{x\in \mathbb{R}}\left[ \delta _{x}\left( t\right) F_{x,s}\left(
t\right) \right] ^{p}\left( \int_{u\in \mathbb{R}}\left[ F_{x,s}\left(u\right) \right] ^{-q}dF\left( u\right) \right) ^{p/q}dF\left( x\right) 
\leq \left( \frac{p}{p-1}\right) ^{p}\text{ .}
\end{equation*}
But we find that 
\begin{eqnarray*}
\lefteqn{\int_{x\in \mathbb{R}}\left[ \delta _{x}\left( t\right) F_{x,s}\left( t\right) \right] ^{p}
    \left( \int_{u\in \mathbb{R}}\left[ F_{x,s}\left( u\right) \right] ^{-q}dF\left( u\right) \right) ^{p/q}dF\left( x\right) }  \nonumber \\
& = & \int_{x\in \mathbb{R}} \left ( 1_{[x, \infty)]} (k) \frac{1_{(-\infty, x]} (t)}{F(x)} 
            +  1_{(-\infty,x)} (k) \frac{1_{(x,\infty)}(t)}{1-F(x)} \right )^p  \nonumber \\
&&  \ \ \cdot \left ( F^s (t) 1_{[x,\infty)} (k) 1_{(-\infty, x]} (t)  + 
             (1-F(t))^s 1_{(-\infty,x)} (k) 1_{(x, \infty )} (t) \right )^p \nonumber  \\
&& \ \ \cdot \left( \int_{u\in \mathbb{R}}\left[ F_{x,s}\left( u\right) \right] ^{-q} dF\left( u\right) \right) ^{p/q} dF\left( x\right) \nonumber \\
& = &  \int_{x \in \mathbb{R}} \left (1_{[x,\infty)} (k) \frac{1_{(-\infty, x]} (t)}{F^p(x)}  F^{sp} (t)  
              +   1_{(-\infty,x)} (k) \frac{1_{(x,\infty)} (t)}{(1-F(x))^p} (1-F(t))^{sp} \right )  \nonumber  \\
&& \ \ \cdot  \left( \int_{u\in \mathbb{R}}\left[ F_{x,s}\left( u\right) \right] ^{-q} dF\left( u\right) \right) ^{p/q} dF\left( x\right)  \nonumber\\ 
& = &  \int_{x \in \mathbb{R}} 1_{[x,\infty)} (k) \frac{1_{(-\infty, x]} (t)}{F^p(x)}  F^{sp} (t) 
             \cdot  \left( \int_{u\in \mathbb{R}}\left[ F_{x,s}\left( u\right) \right] ^{-q} dF\left( u\right) \right) ^{p/q} dF\left( x\right) \nonumber \\ 
&&  \ \ + \   \int_{x \in \mathbb{R}} 
              1_{(-\infty,x)} (k) \frac{1_{(x,\infty)} (t)}{(1-F(x))^p} (1-F(t))^{sp} 
               \cdot  \left( \int_{u\in \mathbb{R}}\left[ F_{x,s}\left( u\right) \right] ^{-q} dF\left( u\right) \right) ^{p/q} dF\left( x\right) \nonumber \\ 
& = & \mathbf{1}_{[t\leq k]} F(t)^{sp} \int_{x=t}^k F (x)^{-p} \left ( \int_{u = - \infty}^x  
                 F(u)^{-sq} dF(u) \right )^{p/q}  dF(x)   \nonumber \\
&& \ \ + \ \mathbf{1}_{[t>k]} (1-F(t))^{sp} \int_{x=k}^t  (1-F(x))^{-p}  \left ( \int_{u = x}^\infty (1-F(u))^{-sq} dF(u) 
                \right )^{p/q} dF(x) \nonumber \\
%& = & \int_{a \vee t}^\infty \frac{F(t)^{sp}}{F(x)^p}  \left ( \int_{-\infty}^x   F(u)^{-sq} dF(u) \right )^{p/q} d F(x) \nonumber \\
%&& \ \ + \ \int_{-\infty}^{k \wedge t} \frac{(1-F(t))^{sp}}{(1-F(x))^{p}} \left ( \int_x^\infty (1-F(u))^{-sq} d F(u) \right )^{p/q} d F(x) \nonumber \\
%& = & \frac{F(t)^{sp}}{(1-sq)^{p/q}} \int_{k \vee t}^\infty  F(x)^{(1-sq)(p/q) - p} dF(x)   \nonumber \\
%&& \ \ + \ \frac{(1-F(t))^{sp}}{(1-sq)^{p/q}} \int_{-\infty}^{k \wedge t} (1-F(x))^{(1-sq)p/q -p} dF(x)  .    \label{keyIdentityGeneralHardy}
%& =&\mathbf{1}_{[t\leq k]}F^{sp}\left( t\right) \int_{x=t}^{k}F\left( x\right)^{-p}
%        \left( \int_{u=-\infty }^{x}F^{-sq}\left( u\right) dF\left( u\right) \right)^{p/q}dF\left( x\right)  \\
%& & +\ \mathbf{1}_{[t>k]}\left( 1-F\left( t\right) \right)^{sp}\int_{x=m}^{t}\left( 1-F\left( x\right) \right) ^{-p}\left(
%      \int_{u=x}^{+\infty }\left( 1-F\left( u\right) \right) ^{-sq}dF\left( u\right) \right)^{p/q}dF\left( x\right)  \\
&=&\mathbf{1}_{[t\leq k]}\left( 1-sq\right) ^{-p/q}F^{sp}\left( t \right)
        \int_{x=t}^{k}F\left( x\right) ^{\frac{p}{q}-sp-p}dF\left( x\right)  \\
& &\ +\ \mathbf{1}_{[t>k]}\left( 1-sq\right) ^{-p/q}\left( 1-F\left( t\right) \right)^{sp}
       \int_{x=k}^{t}\left( 1-F\left( x\right) \right) ^{\frac{p}{q} -sp-p}dF\left( x\right)  \\
& = &\mathbf{1}_{[t\leq k]}\frac{\left( 1-sq\right) ^{-p/q}}{1+p/q-sp-p}
        F^{sp}\left( t\right) \left[ F(k)^{1+p/q-sp-p}-F\left( t\right) ^{1+p/q-sp-p}\right]  \\
&&\ +\ \mathbf{1}_{[t>k]}\frac{\left( 1-sq\right) ^{-p/q}}{1+p/q-sp-p}\left(1-F\left( t\right) \right) ^{sp}
      \left[ (1-F(k))^{1+p/q-sp-p}-\left( 1-F\left( t\right) \right) ^{1+p/q-sp-p}\right]  \\
& = & \mathbf{1}_{[t\leq k]}\frac{\left( 1-sq\right) ^{-p/q}}{sp}
      \left[ 1 - \frac{F^{sp}(t)}{F(k)^{sp}}\right]  \\
&&\ +\ \mathbf{1}_{[t>k]}\frac{\left( 1-sq\right) ^{-p/q}}{sp}
      \left[1 -\frac{(1-F(t))^{sp}}{(1-F(k))^{sp}}\right]   \ \ \ \ \mbox{since} \ 1 + p/q - p = 0\\
&=&\mathbf{1}_{[t\leq k]}\frac{\left( 1-sq\right) ^{-p/q}}{sp}\left( t\right) 
        \left[ 1-\frac{F^{sp}(t)}{F(k)^{sp}}\right]  \\
&&\ +\ \mathbf{1}_{[t>k]}\frac{\left( 1-sq\right) ^{-p/q}}{sp}\left[1 -\frac{(1-F(t))^{sp}}{(1-F(k))^{sp}}\right]  
       \leq \frac{\left( 1-sq\right) ^{-p/q}}{sp}
\end{eqnarray*}

Choosing $s=\left( p-1\right) /p^{2}$  
\begin{equation*}
1-sq=sp=\frac{p-1}{p}>0\ \ ,\ \ 
\end{equation*}
the conclusion for $T_k$ follows.
\end{proof}

$L_{p}$-Poincar\'{e} inequalities are an essential tool of functional analysis in relation 
with the concentration of measure phenomenon. 
In particular, the pathbreaking work of %Milman 
\cite{MR2507637} shows that under 
convexity assumptions that include the log-concave setting, the exponential 
concentration property is equivalent to any $L_{p}$-Poincar\'{e} inequality for $p \in [1,+\infty]$.
Theorem \ref{theorem_lp_lq} induces the following sharp $L_{p}$-Poincar\'{e}
inequalities.

\begin{thm}
\label{Theorem_Lp_Poincare_ineq}Let $\mu $ be a probability measure of finite isoperimetric constant with a
positive density $f$ on $\mathbb{R}$ and cumulative distribution $F$. Let $p \in [1,+\infty)$ and $u\in L_{p}\left(
F\right) $ be absolutely continuous. Then we have,%
\begin{equation}
\left\Vert u-\mathbb{E}\left[ u\right] \right\Vert _{p}\leq 2p\left[ Is\left(
\mu \right) \right] ^{-1}\left\Vert u^{\prime }\right\Vert _{p}\text{ .}
\label{Cheeger_Lp_2}
\end{equation}%
If in addition $\mathbb{E}\left[\sign \left( \left( u-\mathbb{E} \left[ u\right]
\right) \right)\left\vert u-\mathbb{E} \left[ u\right]\right\vert^{p-1} \right]=0$ 
(for instance if $p$ is an even integer and $ \mathbb{E}\left[(u-\mathbb{E} \left[ u\right])^{p-1} \right] = 0$), then
\begin{equation}
\left\Vert u-\mathbb{E}\left[ u\right] \right\Vert _{p}\leq p\left[ Is\left(
\mu \right) \right] ^{-1}\left\Vert u^{\prime }\right\Vert _{p}\text{ .}
\label{Cheeger_Lp_1}
\end{equation}
Furthermore, if $p$ is an odd integer and $\mathbb{E} \left[ u^{p}\right] =0$, then%
\begin{equation}
\left\Vert u\right\Vert _{p}\leq 2p\left[ Is\left( \mu \right) \right]
^{-1}\left\Vert u^{\prime }\right\Vert _{p}\text{ .}
\label{Lp_Poin_mean_0_2}
\end{equation}
Finally, when in addition $\mathbb{E}\left[\sign(u^{p}) \right]=0$, we get 
\begin{equation}
\left\Vert u\right\Vert _{p}\leq p\left[ Is\left( \mu \right) \right]
^{-1}\left\Vert u^{\prime }\right\Vert _{p}\text{ .}
\label{Lp_Poin_mean_0_1}
\end{equation}
\end{thm}

Inequality (\ref{Cheeger_Lp_1}) is a $L_{p}$ version of Cheeger's inequality
(which is recovered by taking $p=2$ in (\ref{Cheeger_Lp_1}) since in this 
case the condition $ \mathbb{E}\left[(u-\mathbb{E} \left[ u\right])^{p-1} \right] = 0$ 
is satisfied for any $u$), see Corollary \ref{corollary_cheeger} above. 
Some general relations between optimal constants in $L_{p}$-Poincar\'{e} 
inequalities are studied in Chapter 2 of \cite{MR2146071}, in connection also with 
constants in some log-Sobolev type inequalities. Estimates of these optimal 
constants through the derivation of Muckenhoupt type criteria are further 
obtained in the framework of Orlicz spaces for dimension one in \cite{MR2146071}, Chapter 5.\\
% (\textbf{Adrien:} we need to ask to a specialist ?). 
Closer to our results, the following $L_p$-Poincar\'{e} inequality is obtained in \cite{BobHoud97}, Remark 7.2: if $p\geq 1$ then
\begin{equation*}
\left\Vert u\right\Vert _{p}\leq 4\sqrt{6}p\left[ Is\left( \mu \right) \right]
^{-1}\left\Vert \nabla u\right\Vert _{p}\text{ ,}
\end{equation*} 
where $\mu$ is a product probability measure, for instance on $\RR^n$, $n\geq1$. 
However, even in the case where $n=1$, it seems that the arguments used in 
\cite{BobHoud97} do not allow to get rid of the constant $4\sqrt{6}$, at least in a simple way.

The Laplace distribution on $\mathbb{R}$, $\epsilon(dx)=1/2 \cdot \exp(-\left\vert x \right\vert) dx$, is 
known to give equality in Cheeger's inequality, see for instance \cite{foug:05}. 
In this case, the isoperimetric constant is equal to one, $Is(\epsilon)=1$. 
Moreover, by applying inequalities (\ref{Cheeger_Lp_1}) and 
(\ref{Lp_Poin_mean_0_1}) for the exponential measure to monomials $u(x)=x^k$ 
and letting 
$k$ be odd and go to infinity, we see that actually, inequalities (\ref{Cheeger_Lp_1}) 
and (\ref{Lp_Poin_mean_0_1}) are optimal in the sense that the ratios 
between the respective left and right sides tend to one as $k$ goes to infinity. 

However, we still do not know if the factors of $2$ in inequalities 
(\ref{Cheeger_Lp_2}) and (\ref{Lp_Poin_mean_0_2}) are necessary or not.

\begin{proof}
Let us begin with the proof of Inequality (\ref{Cheeger_Lp_2}). Define $%
g=u-\mathbb{E} \left[ u\right] $, $h=\sign \left( u-\mathbb{E} \left[ u\right]
\right)$ and take a
sequence $h_{n}$ of Lipschitz functions with values in $\left[ -1,1\right] $
approximating the function $h$ (this sequence exists via Lemma \ref%
{lemma_bob_hou} above). Note that since $(p-1)q=p$, we have $\left\vert g \right\vert^{p-1} \in L_{q}$. 
By the dominated convergence theorem,%
\begin{equation*}
%TCIMACRO{\TeXButton{Cov}{\cov}}%
%BeginExpansion
\cov%
%EndExpansion
\left( g,h_{n}\left\vert g \right\vert ^{p-1}\right) \rightarrow%
%TCIMACRO{\TeXButton{Cov}{\cov}}%
%BeginExpansion
\cov%
%EndExpansion
\left( g,h\left\vert g \right\vert ^{p-1}\right) =\mathbb{E}\left[ \left\vert u-\mathbb{E} \left[ u\right]
\right\vert ^{p}\right] 
\end{equation*}
and
\begin{equation*}
\left\Vert h_{n}\left\vert g \right\vert ^{p-1}-\mathbb{E}\left[ h_{n}\left\vert g \right\vert ^{p-1} \right]\right\Vert
_{q}\rightarrow \left\Vert h\left\vert g \right\vert ^{p-1}-\mathbb{E}\left[ h\left\vert g \right\vert ^{p-1} \right] \right\Vert _{q}%
\end{equation*}%
as $n\rightarrow \infty$. 
Furthermore, Inequality (\ref{ineq_cov_lplq}) of Theorem \ref{theorem_lp_lq}
yields%
\begin{equation*}
%TCIMACRO{\TeXButton{Cov}{\cov}}%
%BeginExpansion
\cov%
%EndExpansion
\left( g,h_{n}\left\vert g \right\vert ^{p-1}\right) \leq p\left[ Is\left( \mu \right) \right]
^{-1}\left\Vert g^{\prime }\right\Vert _{p}\left\Vert h_{n}\left\vert g \right\vert ^{p-1}
-\mathbb{E}\left[ h_{n}\left\vert g \right\vert ^{p-1} \right]\right\Vert_{q}
\end{equation*}%
and taking the limit on both sides we obtain,
\begin{equation}
\left\Vert u-\mathbb{E} \left[ u\right] \right\Vert _{p}^{p}\leq p\left[
Is\left( \mu \right) \right] ^{-1}\left\Vert g^{\prime }\right\Vert_{p}\left\Vert h\left\vert g \right\vert ^{p-1}
-\mathbb{E}\left[ h\left\vert g \right\vert ^{p-1} \right] \right\Vert _{q}\text{ .}  \label{ineq_lim}
\end{equation}
Now,%
\begin{eqnarray}
\lefteqn{\left\Vert h\left\vert g \right\vert ^{p-1}-\mathbb{E}\left[ h\left\vert g \right\vert ^{p-1} \right] \right\Vert _{q}} \notag \\
&\leq & \left\Vert h\left\vert g \right\vert ^{p-1} \right\Vert _{q} 
            + \left\vert \mathbb{E}\left[ h\left\vert g \right\vert ^{p-1} \right] \right\vert \notag \\
& \leq & \mathbb{E}\left[ \left\vert u-\mathbb{E} \left[ u\right] \right\vert ^{(p-1)q}\right] ^{1/q} 
            + \left\vert \mathbb{E}\left[ h\left\vert g \right\vert ^{p-1} \right] \right\vert \notag  \\
& = & \left\Vert u-\mathbb{E}\left[ u\right]\right\Vert _{p}^{p-1}
       +  \left\vert \mathbb{E}\left[\sign  \left( u-\mathbb{E} \left[ u\right]
\right)\left\vert u-\mathbb{E} \left[ u\right]\right\vert^{p-1} \right] \right\vert
  \label{h0_qF_0} \\
&\leq &2\left\Vert u-\mathbb{E}\left[ u\right]
\right\Vert _{p}^{p-1}\text{ .}  \label{h0_qF}
\end{eqnarray}%
Inequality (\ref{Cheeger_Lp_1}) then follows from combining (\ref{ineq_lim})
and (\ref{h0_qF_0}). Inequality (\ref{Cheeger_Lp_2}) is deduced by combining (\ref{ineq_lim})
and (\ref{h0_qF}).

Let us prove Inequalities (\ref{Lp_Poin_mean_0_2}) and (\ref{Lp_Poin_mean_0_1}). Take $g=u^{p}$ and a
sequence $h_{n}$ of Lipschitz functions with values in $\left[ -1,1\right] $
approximating the function $h=$sign$\left( g\right) $ (this sequence exists
via Lemma \ref{lemma_bob_hou} above). By the dominated convergence theorem,%
\begin{equation}
%TCIMACRO{\TeXButton{Cov}{\cov}}%
%BeginExpansion
\cov%
%EndExpansion
\left( g,h_{n}\right) \rightarrow _{n\rightarrow \infty }%
%TCIMACRO{\TeXButton{Cov}{\cov}}%
%BeginExpansion
\cov%
%EndExpansion
\left( g,h\right) =\mathbb{E}\left[ \left\vert u\right\vert ^{p}\right] 
\text{ .}  \label{limit_cov_lp}
\end{equation}%
Furthermore, Inequality (\ref{ineq_cov_lplq}) of Theorem \ref{theorem_lp_lq}
with $p=1$ and $q=+\infty$ yields%
\begin{eqnarray}
\cov \left( g,h_{n}\right) 
&\leq & \left[ Is\left( \mu \right) \right]
^{-1}\left\Vert g^{\prime }\right\Vert _{1} 
(1+\left\vert \mathbb{E}\left[\sign(u^{p}) \right] \right\vert)
 \notag \\
&=& p\left[ Is\left( \mu \right) \right] ^{-1}\left\Vert u^{\prime
}u^{p-1}\right\Vert _{1} (1+\left\vert \mathbb{E}\left[\sign(u^{p}) \right] \right\vert) \notag \\
&\leq & p\left[ Is\left( \mu \right) \right] ^{-1}\left\Vert u^{\prime
}\right\Vert _{p}\left\Vert u^{p-1}\right\Vert _{q} (1+\left\vert \mathbb{E}\left[\sign(u^{p}) \right] \right\vert) \label{cov_l1_0} \\
&\leq & 2p\left[ Is\left( \mu \right) \right] ^{-1}\left\Vert u^{\prime
}\right\Vert _{p}\left\Vert u^{p-1}\right\Vert _{q}
\text{ ,}
\label{cov_l1}
\end{eqnarray}
where the Inequality (\ref{cov_l1_0}) comes from the use of H\"{o}lder's inequality.
Now, noticing that $\left\Vert u^{p-1}\right\Vert _{q}=\left\Vert
u\right\Vert _{p}^{\left( p-1\right) /p}$ and combining (\ref{limit_cov_lp}%
) with (\ref{cov_l1_0}) and (\ref{cov_l1}) we get the results.
\end{proof}

One can notice that Theorem 7.1 of \cite{BobHoud97} establishes more general 
Poincar\'{e} type inequalities for some Orlicz spaces defined with product probability 
measures. $L_p$-norms are special cases of Orlicz norms considered in \cite{BobHoud97}. 
It seems that these Orlicz norms are outside of the scope of our techniques, that are based 
on covariance inequalities established in Section \ref{section_Lp_Cov}. 
However, taking advantage of the dimension one, we can easily provide the following 
Hardy type and Poincar\'{e} type inequalities for Orlicz norms.

Take $\mu $ a probability measure on $\mathbb{R}$ with density $p$ with
respect to the Lebesgue measure and denote by $m$ the median of $\mu $.
Let $N$ be a differentiable convex function and assume furthermore that $N$
is a Young function, meaning that $N$ is even, nonnegative, $N\left(
0\right) =0$ and $N\left( x\right) >0$ for $x\neq 0$. Assume also%
\begin{equation}
C_{N}=\sup \frac{xN^{\prime }\left( x\right) }{N\left( x\right) }<+\infty \label{def_C_N}
\end{equation}%
and set the Orlicz norm associated to $N$ ($N$-norm),%
\begin{equation}
\left\Vert f\right\Vert _{N}=\inf \left\{ \lambda >0:\mathbb{E}\left[
N\left( \frac{f}{\lambda }\right) \right] \leq 1\right\} <+\infty \text{ .} \label{def_Orlicz}
\end{equation}

\begin{prop} \label{prop_Hardy_N_norm}
For a smooth function $f$ with finite $N$-norm, the following Hardy-type
inequality holds,%
\begin{equation}
\left\Vert f - f\left( m\right)\right\Vert _{N} \leq C_{N}Is^{-1}\left( \mu \right) \left\Vert f^{\prime
}\right\Vert _{N}\text{ .} \label{ineq_Hardy_N_norm}
\end{equation}
Consequently,
\begin{equation}
\left\Vert f - \mathbb{E}\left[
f \right]\right\Vert _{N} \leq 2C_{N}Is^{-1}\left( \mu \right) \left\Vert f^{\prime
}\right\Vert _{N}\text{ .} \label{ineq_Poincare_N_norm}
\end{equation}
\end{prop}
One can notice that by taking $N=\left\vert \cdot \right\vert^p$, $p\geq 1$, 
for which $C_N=p$, Inequality (\ref{ineq_Poincare_N_norm}) generalizes 
Inequality (\ref{Cheeger_Lp_2}). However, it seems that a generalization of 
Inequality (\ref{Cheeger_Lp_1}) is not directly accessible by our arguments in the context of Orlicz norms. 
\begin{proof}
Let us prove first Inequality (\ref{ineq_Hardy_N_norm}). 
We can assume without loss of generality that $f$ has a compact support, that $f(m)=0$ and
that $\left\Vert f\right\Vert _{N}=1$. In particular, $\lim_{x \rightarrow -\infty} N(f(x)) F(x) = 0$ and
$\lim_{x \rightarrow \infty} N(f(x))(1- F(x)) = 0$. It follows that%
\begin{eqnarray*}
\mathbb{E}\left[ N\left( f\right) \right]  &=&\int N\left( f\left(
x\right) \right) p\left( x\right) dx \\
&=&\left[ N\left( f\right) F\right] _{-\infty }^{m}-\int_{-\infty
}^{m}f^{\prime }\left( x\right) N^{\prime }\left( f\left( x\right)
\right) F\left( x\right) dx+ \\
&&+\left[ N\left( f\right) \left( F-1\right) \right] _{m}^{+\infty
}-\int_{m}^{+\infty }f^{\prime }\left( x\right) N^{\prime }\left(
f\left( x\right) \right) \left( 1-F\left( x\right) \right) dx \\
&=&\int (-f^{\prime }(x)\mathbf{1}_{(-\infty ,m]}\left( x\right)
+f^{\prime }(x)\mathbf{1}_{[m,\infty )}\left( x\right) )N^{\prime
}(f(x))\min \left\{ \frac{F(x)}{p(x)},\frac{1-F(x)}{p(x)}\right\} d\mu (x)%
\text{ .}
\end{eqnarray*}%
Now, by Lemma 2.1 of \cite{BobHoud97}, if 
\begin{equation*}
\int N\left( g\right) d\mu \leq \int N\left( f\right) d\mu 
\end{equation*}%
then%
\begin{equation*}
\int N^{\prime }\left( f\right) gd\mu \leq \int N^{\prime }\left( f\right)
fd\mu \text{ .}
\end{equation*}%
This directly implies that, in general,%
\begin{equation*}
\left\Vert f\right\Vert _{N}\int N^{\prime }\left( \frac{f}{\left\Vert
f\right\Vert _{N}}\right) gd\mu \leq \left\Vert g\right\Vert _{N}\int
N^{\prime }\left( \frac{f}{\left\Vert f\right\Vert _{N}}\right) fd\mu \text{
.}
\end{equation*}%
Then, by setting $g=(-f^{\prime }\mathbf{1}_{(-\infty ,m]}+f^{\prime }%
\mathbf{1}_{[m,\infty )})\min \left\{ \frac{F}{p},\frac{1-F}{p}%
\right\} $, it follows that%
\begin{eqnarray*}
\mathbb{E}\left[ N\left( f\right) \right] -N(f(m)) = \int N^{\prime}(f)gd\mu  
\leq \left\Vert g\right\Vert _{N}\int N^{\prime }\left( f\right) fd\mu 
\text{ .}
\end{eqnarray*}
Since $\min \left\{ \frac{F}{p},\frac{1-F}{p}\right\} \leq Is^{-1}\left( \mu \right) $ 
almost everywhere and $N$ is even, we have 
$\left\Vert g\right\Vert_{N}\leq Is^{-1}\left( \mu \right) \left\Vert f^{\prime }\right\Vert _{N}$.
In addition,
\begin{eqnarray*}
\int N^{\prime }\left( f\right) fd\mu  
&=&\int N^{\prime }\left( f\right) f \mathbf{1}_{f\leq 0}+N^{\prime }\left( f\right) f\mathbf{1}_{f>0}d\mu  \\
&=&\int -fN^{\prime }\left( -f\right) \mathbf{1}_{f\leq 0}+N^{\prime }\left(
f\right) f\mathbf{1}_{f>0}d\mu  \\
&\leq &C_{N}\int N\left( f\right) d\mu =C_{N}\text{.}
\end{eqnarray*}
This concludes the proof of Inequality (\ref{ineq_Hardy_N_norm}). Then Inequality (\ref{ineq_Poincare_N_norm}) 
is deduced from (\ref{ineq_Hardy_N_norm}) by noticing that
\begin{equation*}
\left\Vert f - \mathbb{E}\left[
f \right]\right\Vert _{N} \leq \left\Vert f - f\left( m\right)\right\Vert _{N} 
+\left\Vert \mathbb{E}[f- f\left( m\right)] \right\Vert _{N} \leq 2\left\Vert f - f\left( m\right)\right\Vert _{N}\text{ .}
\end{equation*}

\end{proof}

To conclude this section, we combine the $L_p$-Poincar{\'e} inequality (\ref{Cheeger_Lp_2}) 
with the covariance inequality given in Theorem \ref{theorem_lp_lq}. 
We obtain the following covariance inequality, which can be compared to inequality (\ref{cov_rho}) - 
which is a consequence of asymmetric Brascamp-Lieb inequality - for strongly log-concave measures.

\begin{prop}
\label{th_cov_LpLq_final}Let $\mu $ be a probability measure with a positive
density $f$ on $\mathbb{R}$ and cumulative distribution $F$. Let $g\in L_{p}\left( F\right) $
and $h\in L_{q}\left( F\right) ,$ $p^{-1}+q^{-1}=1$, $p,q \in (1,+\infty)$. Assume also that $g$
and $h$ are absolutely continuous. Then we have,%
\begin{equation}
\left\vert 
%TCIMACRO{\TeXButton{Cov}{\cov}}%
%BeginExpansion
\cov%
%EndExpansion
( g ,h) \right\vert \leq
2\left( p+q\right) \left[ Is\left( \mu \right) \right] ^{-2}\left\Vert
g^{\prime }\right\Vert _{p}\left\Vert h^{\prime }\right\Vert _{q}\text{ .%
}  \label{cov_ineq_LpLq_final}
\end{equation}
\end{prop}

Notice that using H\"{o}lder's inequality together with Inequality (\ref%
{Cheeger_Lp_2})
%{Cheeger_Lp}) 
of Theorem \ref{Theorem_Lp_Poincare_ineq} already yields,%
\begin{eqnarray*}
\left\vert 
%TCIMACRO{\TeXButton{Cov}{\cov}}%
%BeginExpansion
\cov%
%EndExpansion
\left( g ,h \right) \right\vert &\leq
&\left\Vert g-E_{F}\left[ g\right] \right\Vert _{p}\left\Vert h-E_{F}\left[
h\right] \right\Vert _{p} \\
&\leq &4pq\left[ Is\left( \mu \right) \right] ^{-2}\left\Vert g^{\prime
}\right\Vert _{p}\left\Vert h^{\prime }\right\Vert _{q} \\
&=&4\left( p+q\right) \left[ Is\left( \mu \right) \right] ^{-2}\left\Vert
g^{\prime }\right\Vert _{p}\left\Vert h^{\prime }\right\Vert _{q}\text{ .%
}
\end{eqnarray*}

\begin{proof}
This simply follows by combining Theorem \ref{theorem_lp_lq}\ and Inequality
(\ref{Cheeger_Lp_2}) of Theorem \ref{Theorem_Lp_Poincare_ineq}\ with $u=h$.
\end{proof}

\section{Moment inequalities and concentration\label{section_moment_estimates}}

We investigate here other consequences of covariance inequalities of the
forms obtained in Sections \ref{section_L1_Cov}\ and \ref{section_Lp_Cov}. 
For a measure $\mu $ on $\mathbb{R}$, we will say in this section that $\mu $ 
satisfies a $\left( p,q\right) $ covariance inequality, $p,q\geq 1$ and $%
p^{-1}+q^{-1}=1 $, with constant $c_{p}>0$ if for any absolutely continuous
functions $g\in L_{p}$ and $h\in L_{q}$,%
\begin{equation}
\left\vert 
%TCIMACRO{\TeXButton{Cov}{\cov}}%
%BeginExpansion
\cov%
%EndExpansion
\left( g,h\right) \right\vert \leq c_{p}\left\Vert g^{\prime
}\right\Vert _{p}\left\Vert h_{0}\right\Vert _{q}\text{ .}
\label{cov_LpLq_c}
\end{equation}%
We proved that one can always take $c_{p}=p\left[ Is\left( \mu \right) %
\right] ^{-1}$ (the case where $Is\left( \mu \right) =0$ corresponding to $%
c_{p}=+\infty $ and being uninformative). From such covariance inequalities, 
we can control the growth of moments of $\mu $. 

\begin{prop} \label{theorem_moment_bounds}
Assume that $\mu $ satisfies a $\left( p,q\right) $
covariance inequality with constant $c_{p}>0$ and let 
$X$ be a random variable with distribution $\mu $. Then 
\begin{equation}
\left\Vert X-\bb{E}\left[ X\right] \right\Vert _{p}
:= \left \{ \bb{E} \left[ | X-\bb{E}\left[ X\right] |^{p} \right] \right \}^{1/p}\leq 2c_{p}\text{ .}
\label{bound_moment_p}
\end{equation}%
\end{prop}

For $X$ a real random variable, define the Orlicz norm $\| X \|_{\Psi_{1}} $ %by applying identity 
by (\ref{def_Orlicz}) with $N(v)\equiv\Psi_{1} (v) \equiv \exp (\left\vert v \right\vert) -1$. 
Note that applying (\ref{def_C_N}) with $N=\Psi_{1}$ gives $C_{\Psi_{1}}=+\infty$ 
so that Proposition \ref{prop_Hardy_N_norm} does not apply to the $\Psi_{1}$-norm. 
From Proposition \ref{theorem_moment_bounds}, note also that if there exists $c\in \mathbb{R}_{+}$ such that $c_{p}\leq cp$
then the moments of $X$ grow linearly and 
\begin{equation*}
\left\Vert X-\bb{E}\left[ X\right] \right\Vert _{\Psi _{1}}\leq C\sup_{p>1}\frac{\left\Vert X-\bb{E}\left[ X\right]
\right\Vert _{p}}{p}\leq C\sup_{p>1}\frac{c_{p}}{p}\text{ ,}
\end{equation*}%
for a numerical constant $C>0$. 
\cite{Saumard-Wellner:17} Appendix A show that the inequality 
holds with $C=4$.
If in addition $\sup_{p>1}p^{-1}c_{p}$
is finite then $\mu $ achieves a sub-exponential concentration inequality.
In particular, we recover the well-known fact that if 
$Is\left( \mu \right) >0$ then all moments of $X$ are finite, they grow linearly, $X$ has a finite
exponential moment and achieves a sub-exponential concentration inequality.
More precisely, in this case,%
\begin{equation*}
\left\Vert X-\bb{E}\left[ X\right] \right\Vert _{p}\leq p\left[ Is\left(
\mu \right) \right] ^{-1}\text{ and }\left\Vert X-\bb{E}\left[ X\right]
\right\Vert _{\Psi _{1}\left( \mu \right) }\leq C\left[ Is\left( \mu \right) %
\right] ^{-1}\text{ .}
\end{equation*}

\begin{proof}
It suffices to prove (\ref{bound_moment_p}).  The rest is simple and/or
well-known and thus left to the reader. From (\ref{cov_LpLq_c}) applied with 
$g\left( X\right) =X-\bb{E}\left[ X\right] $ and \\
$h\left( X\right) =$sign$%
\left( X-\bb{E}\left[ X\right] \right) \left\vert X-\bb{E}\left[ X\right]
\right\vert ^{p-1}$, we get%
\begin{eqnarray*}
\left\Vert X-\bb{E}\left[ X\right] \right\Vert _{p} 
& = &  \left\vert  \cov_{\mu }\left( g,h\right) \right\vert \\
&\leq & c_p \left\Vert h_{0}\right\Vert _{q}\leq 2c_p \left\Vert h\right\Vert _{q} \\
&\leq & 2c_p\left\Vert X-\bb{E}\left[ X\right] \right\Vert _{p}^{1-1/p}
\end{eqnarray*}%
and thus (\ref{bound_moment_p}) is proved.
\end{proof}

We have the following moment comparison theorem for measures with finite isoperimetric constant.
\begin{prop}
\label{theorem_moment_comp}
Let $X$ be a real centered random variable of distribution $\mu $ with finite isoperimetric constant and
density $f$ on $\mathbb{R}$. Then, for any $p > 1$,%
\begin{equation*}
\left\Vert  X \right\Vert_{p+1} \leq \left(\frac{p^{2}}{(p-1)Is\left( \mu _{\left\Vert X\right\Vert _{p}}\right)} \right)^{1/(p+1)}
\left\Vert X \right\Vert_{p} \text{ ,}
\end{equation*}%
where $\mu _{c}$ is the distribution associated to the density 
$f_{c}\left( x\right) =cf\left( cx\right) $.
\end{prop}

To the best of our knowledge, the result of Proposition \ref{theorem_moment_comp} is new. 
Its proof is based on the covariance inequality (\ref{ineq_cov_lplq}) of Theorem \ref{theorem_lp_lq}.
On the other hand, Proposition \ref{theorem_moment_comp} is apparently considerably weaker 
than Theorem 8.1 of \cite{BobHoud97} who considered the following setting:
suppose that $v_1, \ldots , v_n$ are vectors in a Banach space $B$, and let $\xi_1, \ldots , \xi_n$ 
be i.i.d. real-valued random variables with $E \xi_1=0$, $\xi_1 \neq 0$ a.s., and with law $\mu$ such that
$Is (\mu) >0$.  Then, with $S \equiv \| \sum_{j=1}^n \xi_j v_j \|_B$ and $N$ a Young function satisfying 
$K_N \equiv \| \Lambda \|_N < \infty$ where $\Lambda $ is the standard Laplace distribution on $\RR$, 
they show that 
\begin{eqnarray}
\| S \|_N \le \left (2 + \frac{8 \sqrt{3} K_N}{Is (\mu) \mathbb{E} | \xi_1 |} \right ) \| S \|_1 .
\label{BobkovHoudre-KK}
\end{eqnarray}
Specializing this to the case $n=1$ and $N(x) = |x|^p$ with $p \ge 1$ yields 
\begin{eqnarray*}
\| \xi_1 \|_{p} \le \left (2 + \frac{8 \sqrt{3} \Gamma (p+1)}{Is (\mu) \mathbb{E} | \xi_1 |} \right ) \| \xi_1 \|_1 .
\end{eqnarray*}
Since our Proposition~ \ref{theorem_moment_comp}  does not allow $p=1$ and since it does 
not yield a Khintchine-Kahane type result as in (\ref{BobkovHoudre-KK}), it is evidently much weaker. 

%Notice that considering linear sums valued in a Banach space of i.i.d. random variables with 
%finite isoperimetric constant, \cite{BobHoud97}, Theorem 8.1, prove Khintchine-Kahane 
%type inequalities that are much more general than Proposition \ref{theorem_moment_comp} 
%and that consist in comparing a Orlicz norm of the Banach norm of the linear sum to its moment of 
%order one. Taking the very special case where the sum has only one element, this gives a 
%comparison of the Orlicz norm of a random variable with finite isoperimetric constant to its 
%first moment. Note that Proposition \ref{theorem_moment_comp} 
%does not allow to obtain comparison to first moment.

\begin{proof}
We use Inequality (\ref{ineq_cov_lplq}) and take $g\left( x\right) =\sign(x)\left\vert x \right\vert^{p}$, 
$h\left( x\right) =x$ and the value $q$ in Inequality (\ref{ineq_cov_lplq}) to be equal to the value of $p$ in 
Theorem \ref{theorem_moment_comp}. Note that we have $g'(x)=p\left\vert x \right\vert^{p-1}$ if $x\neq 0$. This gives%
\begin{eqnarray*}
 \mathbb{E}\left[ \left\vert X \right\vert^{p+1}\right]
   &\leq &\frac{p^{2}}{%
p-1}\left[ Is\left( \mu \right) \right] ^{-1}\mathbb{E}\left[ \left(
\left\vert X \right\vert^{p-1}\right) ^{\frac{p}{p-1}}\right] ^{1-1/p}\mathbb{E}\left[ \left\vert X \right\vert^{p}\right]
^{1/p} \\
&=&\frac{p^{2}}{p-1}\left[ Is\left( \mu \right) \right] ^{-1}\mathbb{E}\left[
\left\vert X \right\vert^{p}\right] \text{ .}
\end{eqnarray*}%
Now the conclusion follows by remarking that by homogeneity, $\left\Vert
X\right\Vert _{p}\left[ Is\left( \mu _{\left\Vert X\right\Vert
_{p}}\right) \right] ^{-1}=\left[ Is\left( \mu \right) \right] ^{-1}$.
\end{proof}

For log-concave measures, Theorem \ref{theorem_moment_comp} has the following corollary.

\begin{cor}
\label{corollary_moment_logcon}
Let $X$ be a real centered log-concave random variable. Then, for any $p \geq 2$,
\begin{equation}
\left\Vert  X \right\Vert_{p+1} \leq \left(\frac{\sqrt{3}p^{2}}{p-1} \right)^{1/(p+1)}
\left\Vert X \right\Vert_{p} \text{ .} \label{moment_log_concave_p}
\end{equation}
In particular,
\begin{equation}
\mathbb{E}\left[ \left\vert X\right\vert^{3}\right] \leq 4\sqrt{3}\mathbb{E}\left[ X^{2}\right]
^{3/2}\text{ .}  \label{moment_log_concave_3_2}
\end{equation}
\end{cor}

\begin{proof}
Simply apply Theorem \ref{theorem_moment_comp} combined 
with the following bound on the isoperimetric constant of a log-concave variable 
(see for instance \cite{BobkovLedoux:14}, Appendix B.2): $Is\left( \mu \right) ^{-1}\leq \sqrt{3%
%TCIMACRO{\TeXButton{Var}{\var}}%
%BeginExpansion
\var%
%EndExpansion
\left( X\right) }$, which implies $Is\left( \mu _{\left\Vert X\right\Vert _{p}}\right) ^{-1}\leq \sqrt{3}$ for any $p \geq 2$.
\end{proof}
\medskip

It is well-known (\cite{MR0388475};  see also \cite{MR2449135,MR3180591,Latala:17}) 
that for $Z$ a real log-concave variable and $p \geq q \geq 2$, 
\begin{equation}
\left\Vert  Z \right\Vert_{p} \leq C \frac{p}{q} 
\left\Vert Z \right\Vert_{q} \text{ ,}  \label{moment_log_concave_p_q}
\end{equation}
where $C=3$ works in general but can be sharpened to $C=1$ for symmetric measures 
and $C=2$ for centered measures. Considering a centered log-concave measure we 
can compare inequality (\ref{moment_log_concave_p}) to inequality 
(\ref{moment_log_concave_p_q}) with $p$ replaced by $p+1$ and $q$ replaced 
by $p$ in (\ref{moment_log_concave_p_q}). 
Then inequality (\ref{moment_log_concave_p}) can be rewritten as
\begin{eqnarray*}
\| X \|_{p+1} \le \left (\frac{p}{p+1} \right ) \left (\frac{\sqrt{3} p^2}{p-1} \right )^{1/(p+1)} \frac{p+1}{p} \| X\|_p 
\equiv C_p \frac{p+1}{p} \| X\|_p
\end{eqnarray*}
where 
\begin{eqnarray*}
C_p = \left (\frac{p}{p+1} \right ) \left (\frac{\sqrt{3} p^2}{p-1} \right )^{1/(p+1)} .
\end{eqnarray*}
The latter constant $C_p$ is a strictly decreasing quantity of $p$ with 
$C_2 = 2^{5/3}/3^{5/6} \approx 1.27091$ and $C_{\infty} =1$.   
In this special case, inequality (\ref{moment_log_concave_p}) indeed 
provides a sharpening of the constant in inequality (\ref{moment_log_concave_p_q}).

Moment comparisons for log-concave vectors is a theme of active research, 
related to conjectures on concentration of log-concave measures 
and involves comparisons of weak and strong moments 
(\cite{MR2276533,MR2918093,MR3180591,MR3150710,MR3503729,Latala:17}).

\cite{bubeck2014entropic} proved an inequality related to (\ref{moment_log_concave_3_2}), 
with the absolute moment $\mathbb{E}[ \left\vert X\right\vert^{3}] $ in (\ref{moment_log_concave_3_2}) 
replaced by the smaller third moment $\mathbb{E} [  X^{3} ] $ and with $4\sqrt{3}$ replaced by $2$.
Their original proof of this result 
%is quite tedious and 
is based on a rather deep result of %Fradelizi and Gu\'{e}don 
\cite{FradGued06} 
describing the extremal points of a convex set of log-concave measures. 
In comparison, the derivation of Theorem \ref{theorem_moment_comp} and 
Corollary \ref{corollary_moment_logcon} from the covariance inequality 
(\ref{ineq_cov_lplq}) is much easier.

\section*{Acknowledgements}

We owe thanks to Pierre Foug\`{e}res and Christian Houdr\'{e} for a number of pointers to the literature. 
We are also grateful to a referee who pointed us to the fact that Hoeffding's covariance identity actually 
implies Menz and Otto's covariance representation. His remarks also greatly improved the presentation of the paper.

\bibliographystyle{imsart-nameyear}
\bibliography{chern}

\end{document}